\documentclass[10pt,a4paper]{article}
\usepackage[utf8]{inputenc}
\usepackage[english]{babel}
\usepackage[a4paper, margin=25mm]{geometry}
\usepackage{mathtools}
\usepackage{dsfont}
\usepackage{graphicx}
\usepackage{pdfpages}
\usepackage{epigraph}
\usepackage{fancyhdr}
\usepackage{physics}
\usepackage{float}
\usepackage{caption}
\usepackage{tikz-cd}
\usepackage{subcaption}
\usepackage{tikz}
\usepackage{svg}
\usepackage{tcolorbox}
\usepackage{amsthm}
\usepackage{amssymb}
\usepackage{amsfonts}
\usepackage{hyperref}
\usepackage{cleveref}
\usepackage{lipsum}
\usepackage{enumitem}
\usepackage{tensor}
\usepackage{tcolorbox}
\usepackage{mathrsfs}
\usepackage{titling}
\usepackage{dsfont}
\usepackage[runin]{abstract}
\usepackage{appendix}

\hypersetup{
    colorlinks=,
    linkcolor={red!50!black},
    citecolor={blue!50!black},
    urlcolor={blue!80!black},
    pdfborder={1 1 1}
}

\abslabeldelim{:}

\theoremstyle{definition}
\newtheorem{definition}{Definition}
\newtheorem{theorem}{Theorem}
\newtheorem{corollary}[theorem]{Corollary}
\newtheorem{lemma}{Lemma}
\newtheorem*{theorem*}{Theorem}

\numberwithin{definition}{section}
\numberwithin{theorem}{section}
\numberwithin{lemma}{section}

\makeatletter
\newcommand\RedeclareMathOperator{%
  \@ifstar{\def\rmo@s{m}\rmo@redeclare}{\def\rmo@s{o}\rmo@redeclare}%
}
\newcommand\rmo@redeclare[2]{%
  \begingroup \escapechar\m@ne\xdef\@gtempa{{\string#1}}\endgroup
  \expandafter\@ifundefined\@gtempa
     {\@latex@error{\noexpand#1undefined}\@ehc}%
     \relax
  \expandafter\rmo@declmathop\rmo@s{#1}{#2}}
\newcommand\rmo@declmathop[3]{%
  \DeclareRobustCommand{#2}{\qopname\newmcodes@#1{#3}}%
}
\@onlypreamble\RedeclareMathOperator
\makeatother

\DeclareMathOperator{\N}{\mathbb{N}}
\DeclareMathOperator{\R}{\mathbb{R}}
\DeclareMathOperator{\Z}{\mathbb{Z}}

\RedeclareMathOperator{\L}{\mathcal{L}}

\RedeclareMathOperator{\S}{\mathbb{S}}
\RedeclareMathOperator{\P}{\mathbb{P}}

\RedeclareMathOperator{\O}{\mathcal{O}}

\newcommand{\innp}[2]{\left\langle #1, #2 \right\rangle}

\title{\vspace{-0.5in}{\bfseries\scshape 
    Towards an Approximation Theory of Observable Operator Models
    }
}

\author{
    \normalsize{Wojciech Anyszka, \textit{University of Groningen}}
}

\date{\vspace{-5ex}}

\begin{document}

\maketitle
\begin{abstract}
{Observable operator models (OOMs) offer a powerful framework for modelling stochastic processes, surpassing the traditional hidden Markov models (HMMs) in generality and efficiency. However, using OOMs to model infinite-dimensional processes poses significant theoretical challenges. This thesis presents an exploration of a rigorous approach to developing an approximation theory for OOMs of infinite-dimensional processes. Building upon foundational work outlined in an unpublished tutorial \cite{unpub-tutorial}, an inner product structure on the space of future distributions is rigorously established and the continuity of observable operators with respect to the associated 2-norm is proven. The original theorem proven in this thesis describes a fundamental obstacle in making an infinite-dimensional space of future distributions into a Hilbert space. The presented findings lay the groundwork for future research in approximating observable operators of infinite-dimensional processes, while a remedy to the encountered obstacle is suggested. \vspace{6ex}}
\end{abstract}

{\bf Note:} This document is a copy of a thesis completed as part of the BSc in Artificial Intelligence at the University of Groningen, with the original available at the university's repository \cite{my_thesis}. This thesis was supervised by Herbert Jaeger.

\section{Introduction}\label{sec:introduction}

 Observable operator models (OOMs) are mathematical models of stochastic processes that properly generalize hidden Markov models (HMMs) \cite{first-OOM}. Despite being more general than HMMs, their learning algorithm is more efficient. This leads to OOMs outperforming HMMs on a variety of datasets \cite{ESAlg-paper} and being used in applications ranging from robotics \cite{OOMsInRobotics} to chemistry \cite{chemistry-OOM}.

 The main idea of OOMs is to identify updates due to new observations with operators acting on a vector space spanned by future distributions of the process. The dimension of this vector space is called the dimension of the stochastic process. Intuitively, the larger this dimension is the more information each observation contains. It is expected that, due to their complexity, most real-world stochastic processes will have infinite dimensions. However, such a process cannot be directly modelled on a computer. Consequently, there is a need to develop a theory which would give theoretical guarantees that allow approximating an infinite-dimensional process with finite-dimensional ones.

 In the case of discrete-time, stationary stochastic processes with a countable set of possible observations, the beginning of such a theory has been developed in an unpublished tutorial \cite{unpub-tutorial}. In this tutorial, the space of future distributions is equipped with an inner product and the observable operators are proven to be continuous with respect to the associated $2$-norm. The hope is that if this inner product space is a Hilbert space, one would have at least two pathways for approximating observable operators. The first one is to prove that the observable operators are compact and then use the approximation property of Hilbert spaces to approximate them with finite-rank operators. The second pathway is to prove that the Hilbert space is separable, and then the observable operators can be approximated (in the strong-operator topology) by finite-rank operators. Both of those pathways would allow the development of theoretical guarantees for the desired approximation.

The aforementioned unpublished tutorial \cite{unpub-tutorial} contains only the general idea of the construction and all of the results are without proof. Unfortunately, these important details have been lost. This thesis has two research objectives. The first one is to recover all the mathematical details of the construction presented in the unpublished tutorial as well as reinvent the missing proofs. The second objective is to investigate under what conditions the space of future distributions together with the inner product constructed in \cite{unpub-tutorial} is a Hilbert space. This is an open question and answering it is the first step towards developing an approximation theory for OOMs via the two outlined pathways.

The thesis is structured as follows. In Section 2, the necessary concepts from stochastic processes are reviewed and the used notation is set. Section 3 introduces abstract OOMs and their basic well-known properties. The fourth section rigorously introduces an inner-product structure on the space of future distributions. Finally, in Section 5, the reinvented proofs of various previously known results regarding this inner product space are shown and the main original theorem --- addressing the second research objective --- is established. It states the aforementioned inner product space is a Hilbert space if and only if it is finite-dimensional. In the final section, conclusions are drawn and directions for future research are indicated. The measure-theoretic notions used in this thesis are briefly summarized in Appendix A.
\section{Stochastic Processes}

In this section, the basic notions and notations are introduced for a discrete-time, stationary stochastic process with a countable set of observations $\O$. 

The individual observations will be denoted by lowercase letters with natural number subscript, i.e. $\O=\{a_{i}\mid i\in \N\}$. A discrete-time stochastic process with values in $\O$ is a collection of random variables $(X_{t})_{t\in \Z}$ each of which takes values in the set $\O$. These random variables can be thought to represent the observation outcome at the corresponding time $t\in \Z$. The probability of outcome $a_{i_{1}}$ at $t_{1}$,...,  $a_{i_{k}}$ at ${t_{k}}$ will be denoted by
\begin{align*}
    \P(X_{t_{1}}=a_{i_{1}},...,X_{t_{k}}=a_{i_{k}}).
\end{align*}
As already mentioned, this thesis will only concern stationary discrete-time stochastic processes. These have an additional property that 
\begin{align*}
    \P(X_{t_{1}}=a_{i_{1}},...,X_{t_{k}}=a_{i_{k}})=\P(X_{t_{1}+l}=a_{i_{1}},...,X_{t_{k}+l}=a_{i_{k}})
\end{align*}
for any integer $l$. In other words, a stationary stochastic process has a time translation invariance. Because of this feature, when outcomes at consecutive times steps are examined the reference to random variables can be dropped, i.e.
\begin{align*}
    \P(a_{i_{1}}\ldots a_{i_{k}}) := \P(X_{t_{1}}=a_{i_{1}},\ldots, X_{t_{1}+k-1}=a_{i_{k}}).
\end{align*}
An arbitrary sequence of consecutive outcomes will be denoted by lowercase letters with a bar, i.e. $\bar{a}\in \O^{k}$ for some $k\in \N$. Using this notation, the fact that some sequence of length $k$ must occur can be expressed by
\begin{align*}
    \sum_{\bar{a}\in\O^{k}} \P(\bar{a}) = 1.
\end{align*}
Similarly, the probability of a sequence $\bar{a}$ conditioned on a directly preceding sequence $\bar{b}=(a_{j_{1}},...,a_{j_{m}})$ will be denoted by
\begin{align*}
    \P(\bar{a}\mid\bar{b}):= \P(X_{t_{1}}=a_{i_{1}},\ldots, X_{t_{1}+k-1}=a_{i_{k}} \mid X_{t_{1}-m} = a_{j_{1}},\ldots, X_{t_{1}-1} = a_{j_{m}}).
\end{align*}
With this notation, the formula for conditional probability can be written as
\begin{align*}
    \P(\bar{b}\bar{a}) = \P(\bar{a}\mid \bar{b})\P(\bar{b}),
\end{align*}
where $\bar{b}\bar{a}$ denotes concatenation of sequences. Finally, a lowercase letter with a tilde will denote an infinite sequence of symbols, e.g. $\tilde{a}=a_{1}a_{2}\ldots$, while $\tilde{a}^{k} := a_{1}a_{2}\ldots a_{k}$.

\section{Abstract OOMs}

With the notation set, let us introduce abstract OOMs for modelling a discrete-time, stationary stochastic process $(X_{t})_{t\in \Z}$ with a countable set of observations $\O$. For the concrete representations of OOMs and their relation to HMMs, the reader is advised to consult \cite{ESAlg-paper}.

First, let us set $\O^{+}:=\bigcup_{i=1}^{\infty}\O^{i}$ and $\O^{\ast}:=\O^{+} \cup \{\epsilon\}$, where $\epsilon$ denotes the empty string. The observable operators will be acting on the vector space spanned by the functions $g_{\bar{a}}: \O^{\ast}\to \R $ defined for each $\bar{a}\in \O^{\ast}$ by
\begin{align*}
    g_{\bar{a}}(\bar{b}) := \begin{cases}
        \P(\bar{b}\mid\bar{a}) & \text{if } \P(\bar{a}) \neq 0\\
        0 &\text{otherwise,}
    \end{cases}
\end{align*}
with the convention that $g_{\bar{a}}(\epsilon)=1$ if $P(\bar{a})\neq 0$, otherwise it is $0$. Intuitively, the function $g_{\bar{a}}$ describes the distribution of a process given a previous observation $\bar{a}$. We denote the vector space --- with addition and scalar multiplication defined pointwise --- spanned by such functions $g_{\bar{a}}$ by $\mathcal{G}$ and call it the {\bf space of future distributions}. A crucial observation at this point is that even though $\mathcal{G}$ can be seen as a subspace of the vector space of real-valued functions on $\O^{\ast}$, this ambient space does not have an intrinsic topology and so taking infinite linear combinations of functions $g_{\bar{a}}$ does not make sense. Therefore $\mathcal{G}$ is the space of all finite linear combinations of these functions. Since every spanning set can be reduced to a basis, there exists a subset $\Lambda\subseteq \O^{\ast}$ such that $(g_{\bar{a}})_{\bar{a}\in \Lambda}$ is a Hamel basis for $\mathcal{G}$.

For each $a\in \O$, an observable operator $t_{a}:\mathcal{G}\to\mathcal{G}$ is defined by
\begin{align*}
    t_{a}(g_{\bar{c}})=\P(a\mid\bar{c})g_{\bar{c}a}
\end{align*}
for each $\bar{c}\in \Lambda$ and extended linearly to other elements of $\mathcal{G}$. This definition does not depend on the choice of basis $\Lambda$ as long as the basis elements are all of the form $g_{\bar{a}}$ for some $\bar{a}\in \O^{\ast}$. This follows from the fact that for any $\bar{b}\in \O^{\ast}$ with $g_{\bar{b}}=\sum_{\bar{c}\in\Lambda} \lambda_{\bar{c}} g_{\bar{c}}$ and $\P(\bar{b})\neq 0$, it holds that
\begin{align*}
    t_{a}(g_{\bar{b}})(w) &= \sum_{\bar{c}\in\Lambda} \lambda_{\bar{c}} t_{a}(g_{\bar{c}})(w)= \sum_{\substack{\bar{c}\in\Lambda,\\ \P(\bar{c}a)\neq 0}} \lambda_{\bar{c}} \P(a\mid\bar{c})\P(w\mid \bar{c}a)= \sum_{\substack{\bar{c}\in\Lambda,\\ \P(\bar{c}a)\neq 0}} \lambda_{\bar{c}} \P(aw\mid\bar{c}).
\end{align*}
But if $\P(\bar{c}a)=0$ then $\P(aw\mid \bar{c})=0$ and so
\begin{align*}
    t_{a}(g_{\bar{b}})(w)=\sum_{\substack{\bar{c}\in\Lambda}} \lambda_{\bar{c}} \P(aw\mid\bar{c})=\sum_{\substack{\bar{c}\in\Lambda}} \lambda_{\bar{c}} g_{\bar{c}}(aw)=g_{\bar{b}}(aw)=\P(aw\mid \bar{b}) = \P(\bar{a}  \mid\bar{b})\P(w\mid \bar{b}a)= \P(\bar{a}  \mid\bar{b})g_{\bar{b}a}(w),
\end{align*}
where we used that $\P(\bar{b})\neq 0$. Therefore $t_{a}(g_{\bar{b}})=\P(a\mid\bar{b})g_{\bar{b}a}$ for all $\bar{b}\in \O^{\ast}$ with $\P(\bar{b}) \neq 0$. Moreover, if $\P(\bar{b})=0$ then $g_{\bar{b}}\equiv 0$ which implies that $g_{\bar{b}}$ cannot be part of any basis of $\mathcal{G}$. Hence it follows that another choice of basis of the form $\{g_{\bar{a}}\}_{\bar{a}\in \Gamma}$ with $\Gamma\subseteq \O^{\ast}$ would lead to the same operator. Finally, define a functional $\sigma$ as the evaluation at $\epsilon$. Using the definition of $g_{\bar{a}}$ it follows that
\begin{align*}
    \sigma(g_{\bar{a}})=g_{\bar{a}}(\epsilon)=\begin{cases}
        1 &\text{if }\P(\bar{a})\neq 0\\
        0 &\text{else.}
    \end{cases}
\end{align*}

The quadruple $(\{t_{a}\}_{a\in\Lambda}, \sigma, g_{\epsilon}, \mathcal{G})$ is called an OOM of $(X_{t})_{t\in\Z}$. The dimension of the process is defined to be the vector space dimension of $\mathcal{G}$. Now, an OOM fully characterizes the underlying stochastic process $(X_{t})_{t\in\Z}$. This is a consequence of the following fundamental theorem of OOMs.
\begin{theorem}\label{OOM_Fun - thm}
\begin{align*}
    \P(a_{1}...a_{n})=\sigma(t_{a_{n}}...t_{a_{1}}g_{\epsilon}).
\end{align*}
\end{theorem}
\begin{proof}
    \begin{align*}
    \sigma(t_{a_{n}}...t_{a_{1}}g_{\epsilon})&=\sigma(t_{a_{n}}...\P(a_{1}\mid \epsilon)g_{a_{1}})\\
    &=\P(a_{1})\sigma(t_{a_{n}}...t_{a_{2}}g_{a_{1}})\\
    &=\P(a_{2}\mid a_{1})\P(a_{1})\sigma(t_{a_{n}}...t_{a_{3}}g_{a_{1}a_{2}})\\
    &= \sigma(g_{a_{1}a_{2}...a_{n}})\P(a_{n}\mid a_{1}...a_{n-1})...\P(a_{2}\mid a_{1})\P(a_{1})\\
    &= \P(a_{1}...a_{n}),
\end{align*}
where it was used that $\sigma(g_{\overline{b}})=1$ for all $\overline{b}$ with $\P(\bar{b})\neq0$ and otherwise it is $0$. 
\end{proof}

Therefore, to model a stochastic process $(X_{t})_{t\in \Z}$ it suffices to learn a corresponding OOM. For finite-dimensional processes there is an efficient learning algorithm to learn an OOM, it is called Efficiency Sharpening (ES) \cite{ESAlg-paper}. The goal of this thesis is to contribute to the development of an approximation theory for infinite-dimensional processes. 

\section{The Inner Product Construction}

As was outlined in the introduction, the starting point for developing an approximation theory is to endow the space $\mathcal{G}$ with an inner product. To this end, let $F=\{a_{1}a_{2}\ldots \mid a_{i}\in \O\}$ be the set of all right infinite sequences of symbols. For $b_{1},\ldots, b_{k}\in \O$ define $E(b_{1},...b_{k})\subseteq F$ to be
\begin{align*}
    E(b_{1},\ldots, b_{k}) := \{a_{1}a_{2}\ldots \in F \mid a_{1}=b_{1}, \ldots, a_{k} = b_{k}\}.
\end{align*}
Now set
\begin{align*}
    \mathcal{E} :=  \bigcup_{k=1}^{\infty} \left\{ \bigcup_{\bar{b}\in K} E(\bar{b}) \mid K\subseteq \O^{k}\right\},
\end{align*}
i.e. $\mathcal{E}$ is the set of all homogeneous unions of sets $E(\bar{b})$, where homogeneity means that each union is over sequences of the same length.
\begin{lemma}\label{algebra}
    $\mathcal{E}$ is a Boolean algebra over $F$.
\end{lemma}
\begin{proof}
    \begin{enumerate}
        \item Take any $S= \bigcup_{\bar{b}\in K} E(\bar{b})\in \mathcal{E}$ with $K\subseteq \O^{k}$. Since $\O$ is countable so is $\O^{k}$. Consequently, the set $U =  \O^{k}\setminus K$ is countable as well. Noting that for $\bar{b}_{1},\bar{b}_{2}\in \O^{k}$, the sets $E(\bar{b}_{1}), E(\bar{b}_{2})$ are disjoint if and only if $\bar{b}_{1}\neq \bar{b}_{2}$, it follows that the complement of $S$ in $F$ is $S^{c} = \bigcup_{\bar{b}\in U}E(\bar{b})$. Clearly, $S^{c}$ belongs to $\mathcal{E}$ and so $\mathcal{E}$ is closed under complements.
        \item $F\in \mathcal{E}$ since $F=\bigcup_{\bar{b}\in \O^{k}} E(\bar{b})$.
        \item Finally, observe that for $\bar{b}\in\O^{k}, \bar{a}\in \O^{l}$ with $k\geq l$, it holds that
        \begin{align*}
            E(\bar{b})\cap E(\bar{a}) = \begin{cases}
                E(\bar{b}) &\text{if } \bar{b}=\bar{a}\bar{c}\text{ for some } \bar{c}\in \O^{k-l}\\
                \emptyset &\text{else,}
            \end{cases}
        \end{align*}
        where $\O^{0}:= \{\epsilon\}$. As $\emptyset = F^{c}\in \mathcal{E}$, it follows that $E(\bar{b})\cap E(\bar{a}) \in \mathcal{E}$. Now take any $S, T\in \mathcal{E}$, then
        \begin{align*}
            S\cap T = \left(\bigcup_{\bar{c}\in K} E(\bar{c})\right)\cap \left(\bigcup_{\bar{a}\in M} E(\bar{a})\right) = \bigcup_{\bar{c}\in K}\bigcup_{\bar{a}\in M} E(\bar{c})\cap E(\bar{a}),
        \end{align*}
        for some $K\subseteq \O^{k}, M\subseteq\O^{l}$.  Without loss of generality assume that $k\geq l$, thus each $E(\bar{c})\cap E(\bar{a})$ either equals $E(\bar{c})$ or $\emptyset$. Thus, 
        defining
        \begin{align*}
            L := \{\bar{c}\in K \mid E(\bar{c})\cap E(\bar{a})=E(\bar{c}) \text{ for some } \bar{a}\in M\}
        \end{align*}
        it holds that
        \begin{align*}
            S\cap T = \bigcup_{\bar{c}\in L}E(\bar{c}).
        \end{align*}
        From this, it follows that $S\cap T\in \mathcal{E}$ which concludes the proof.
    \end{enumerate}
\end{proof}

Next, for each $\bar{a}\in \O^{\ast}$, define a set function $\bar{\mu}_{\bar{a}}:\mathcal{E} \to [0,\infty]$ by
\begin{align*}
    \bar{\mu}_{\bar{a}}\left(\bigcup_{\bar{b}\in K} E(\bar{b})\right) := \sum_{\bar{b}\in K}g_{\bar{a}}(\bar{b})
\end{align*}
and set $\bar{\mu}_{\bar{a}}\left(\emptyset\right) := 0$. Note that the above series always converges (absolutely) since
\begin{align*}
    \sum_{\bar{b}\in K}g_{\bar{a}}(\bar{b}) \leq \sum_{\bar{b}\in \O^{k}}g_{\bar{a}}(\bar{b}) = 1
\end{align*}
and all of the summands are positive. The last line of the above computation also shows that $\bar{\mu}_{\bar{a}}(F)=1$. The goal is to prove that $\bar{\mu}_{\bar{a}}$ is a pre-measure on $\mathcal{E}$. For that, the following set-theoretic result will be needed.
\begin{lemma}\label{partition-lemma}
    Let $S$ be a set. Suppose $\{A_{i}\}_{i\in I}, \{B_{j}\}_{j\in J}$ are partitions of $S$ with the property that for each $i\in I$ and $j\in J$, $A_{i}$ and $B_{j}$ are either disjoint or one is the subset of the other. Then, there exists a partition of $I$ into sets $\{I_{0}\}\cup \{I_{\lambda}\}_{\lambda\in J_{0}}$, i.e. $I=I_{0}\cup\; \bigcup_{\lambda\in J_{0}}I_{\lambda}$, and a partition of $J$ into sets $\{J_{0}\}\cup \{J_{\gamma}\}_{\gamma\in I_{0}}$ with the following properties
    \begin{enumerate}
        \item For each $i\in I_{0}$
        \begin{align*}
            A_{i} = \bigcup_{j\in J_{i}} B_{j}
        \end{align*}
        \item For each $j\in J_{0}$
        \begin{align*}
            B_{j}  = \bigcup_{i\in I_{j}} A_{i}
        \end{align*}
    \end{enumerate}
 \end{lemma}
\begin{proof}
    Consider that for $i\in I$,
    \begin{align*}
        A_{i} = A_{i}\cap S = A_{i} \cap \bigcup_{j\in J} B_{j} = \bigcup_{j\in J} A_{i}\cap B_{j}. 
    \end{align*}
    Since $A_{i}$ is nonempty, there must be an index $j_{0}\in J$ such that $A_{i}\cap B_{j_{0}}$ is nonempty. From the assumptions, it follows that either $A_{i}\subsetneq B_{j_{0}}$ or $A_{i}\supseteq B_{j_{0}}$. Now, let us define $I_{0}$ by
    \begin{align*}
        I_{0} := \{ i\in I \mid \text{there exists } j\in J \text{ s.t. } A_{i}\supseteq B_{j}\}.
    \end{align*}
    and for any $i\in I_{0}$ define
    \begin{align*}
        J_{i} := \{j\in J \mid B_{j}\cap A_{i} \neq \emptyset\}.
    \end{align*}
    Let us make the corresponding definitions for the $\{B_{j}\}$ partition. Namely,
    \begin{align*}
        J_{0} := \{ j\in J \mid \text{there exists } i\in I \text{ s.t. } A_{i}\subsetneq B_{j}\},
    \end{align*}
    and for each $j\in J_{0}$
    \begin{align*}
        I_{j} := \{ i\in I \mid B_{j}\cap A_{i} \neq \emptyset\}.
    \end{align*}
    The claim is that $\{I_{0}\} \cup \{I_{\lambda}\}_{\lambda\in J_{0}}$ and $\{J_{0}\} \cup \{J_{\gamma}\}_{\gamma\in I_{0}}$ are partitions of $I$ and $J$, respectively, with the desired properties. First, let us show that the collection of sets $\{I_{0}\} \cup \{I_{\lambda}\}_{\lambda\in J_{0}}$ partitions $I$. Well indeed, take any $i\in I$ then there must be an index $\hat{j}\in J$ such that $A_{i}\cap B_{\hat{j}}$ is nonempty. From the assumptions, it is known that either $A_{i}\subsetneq B_{\hat{j}}$ or $A_{i}\supseteq B_{\hat{j}}$. If $A_{i}\subsetneq B_{\hat{j}}$ then clearly $i\notin I_{0}$. Additionally, $\hat{j}\in J_{0}$ and hence $i\in I_{\hat{j}}$, since $\{B_{j}\}$ partitions $S$ and $A_{i}\subsetneq B_{\hat{j}}$, this $\hat{j}\in J$ is unique. This proves that if $A_{i}\subsetneq B_{\hat{j}}$ then $i$ belongs to a unique set from the collection $\{I_{0}\} \cup \{I_{\lambda}\}_{\lambda\in I_{0}}$. On the other hand, if $A_{i}\supseteq B_{\hat{j}}$ then $i\in I_{0}$. Moreover, $i\notin I_{j}$ for all $j\in J_{0}$. To see this, suppose towards a contradiction that $i\in I_{j}$ for some $j\in J_{0}$ but then $A_{i}\cap B_{j}\neq \emptyset$. This would, in turn, imply that there is $\tilde{i}\in I$ such that $A_{\tilde{i}}\subsetneq B_{j}$ and since $\{A_{i}\}$ partition $S$ it would follow that $A_{i}\subsetneq B_{j}$\footnote{Recall that there are only three options: $A_{i}\cap B_{\hat{j}}= \emptyset$ or  $A_{i}\subsetneq B_{\hat{j}}$ or $A_{i}\supseteq B_{\hat{j}}$.} which is a contradiction. Consequently, each $i\in I$ belongs to a unique set from the collection $\{I_{0}\} \cup \{I_{\lambda}\}_{\lambda\in I_{0}}$ and hence this collection partitions $I$. An analogous argument proves that $\{J_{0}\} \cup \{J_{\gamma}\}_{\gamma\in I_{0}}$ partitions $J$. Finally, let us check the desired properties
    \begin{enumerate}[label=(\roman*)]
        \item Take $i\in I_{0}$ then
        \begin{align*}
        A_{i} = A_{i}\cap S= A_{i} \cap \bigcup_{j\in J} B_{j} = \bigcup_{j\in J} A_{i}\cap B_{j} =  \bigcup_{j\in J_{i}} A_{i}\cap B_{j}. 
    \end{align*}
    \item Take $j\in J_{0}$, then
    \begin{align*}
        B_{j}= B_{j}\cap S= B_{j} \cap \bigcup_{i\in I} A_{i} = \bigcup_{i\in I} B_{j}\cap A_{i} =  \bigcup_{i\in I_{j}} B_{j}\cap A_{i}.
    \end{align*}
    \end{enumerate}
\end{proof}

\begin{lemma}\label{premeasure}
    For any $\bar{a}\in \O^{\ast}$, the set function $\bar{\mu}_{\bar{a}}$ is a finite pre-measure on $\mathcal{E}$.  
\end{lemma}
\begin{proof} 
    By definition, $\bar{\mu}_{\bar{a}}(\emptyset)=0$. Now, take a disjoint family $\{A_{n}\mid n\in \N\}$ of elements of $\mathcal{E}$ such that $\bigcup_{i=1}^{\infty}A_{n}\in \mathcal{E}$. Take an arbitrary set $A_{n}$. Since $A_{n}\in \mathcal{E}$ it follows that
   \begin{align*}
       A_{n} = \bigcup_{\bar{b}\in K_{n}}E(\bar{b}),
   \end{align*}
   for some $K_{n}\subseteq \mathcal{O}^{k_{n}}$. Note that for different $\bar{b},\bar{c}\in K_{n}$, $E(\bar{b})$ and $E(\bar{c})$ are disjoint. Moreover, since for $m\neq n$,  $A_{n}$ and $A_{m}$ are disjoint, it also holds that $E(\bar{d})$ and $E(\bar{e})$ are disjoint for all $\bar{d}\in K_{n}, \bar{e}\in K_{m}$. Consequently, all of the sets $E(\bar{b})$ for $\bar{b} \in \bigcup_{n\in \N} K_{n}$ are pairwise disjoint. Now, since $\bigcup_{i=1}^{\infty}A_{n}\in \mathcal{E}$ it follows that there exists $L\subseteq \mathcal{O}^{l}$ for some $l\in \N$ such that
   \begin{align*}
       \bigcup_{n=1}^{\infty}A_{n} = \bigcup_{\bar{a}\in L}E(\bar{a}).
   \end{align*}
   On the other hand,
   \begin{align*}
        \bigcup_{n=1}^{\infty}A_{n}=\bigcup_{n=1}^{\infty}\bigcup_{\bar{b}\in K_{n}}E(\bar{b})=\bigcup_{\bar{b}\in \cup_{i\in\N}K_{i}}E(\bar{b}).
   \end{align*}
    Both of the above give a decomposition of $\bigcup_{i=1}^{\infty}A_{n}$ into disjoint sets. Moreover, since the sets $E(\bar{a})$ and $E(\bar{b})$ are always either disjoint or one is contained in the other, we can apply \Cref{partition-lemma} to obtain a partition of $L$ into $\{L_{0}\} \cup \{L_{\bar{b}}\}_{\bar{b}\in N_{0}}$ and a partition of $N:=\cup_{i\in\N}K_{i}$ into $\{N_{0}\} \cup \{N_{\bar{c}}\}_{\bar{c}\in L_{0}}$. Using the partition $L= L_{0}\cup\;\bigcup_{\bar{b}\in N_{0}} L_{\bar{b}}$ and the definition of $\bar{\mu}_{\bar{a}}$ gives
    \begin{align*}
       \bar{\mu}_{\bar{a}}\left(\bigcup_{n=1}^{\infty}A_{n}\right) &=  \bar{\mu}_{\bar{a}}\left(\bigcup_{\bar{b}\in L}E(\bar{b})\right) = \sum_{\bar{b}\in L} \bar{\mu}_{\bar{a}}(E(\bar{b}))  = \sum_{\bar{c}\in L_{0}} \bar{\mu}_{\bar{a}}(E(\bar{c})) + \sum_{\bar{b}\in N_{0}}\sum_{\bar{c}\in L_{\bar{b}}} \bar{\mu}_{\bar{a}}(E(\bar{c})).
    \end{align*}
    Applying property (i) from \Cref{partition-lemma} to the first term and using disjointedness of $E(\bar{c})$'s in the second one yields
    \begin{align*}    
          \bar{\mu}_{\bar{a}}\left(\bigcup_{n=1}^{\infty}A_{n}\right) &= \sum_{\bar{c}\in L_{0}} \bar{\mu}_{\bar{a}}(\cup_{\bar{b}\in N_{\bar{c}}}E(\bar{b})) + \sum_{\bar{b}\in N_{0}} \bar{\mu}_{\bar{a}}(\cup_{\bar{c}\in L_{\bar{b}}}E(\bar{c})).
    \end{align*}
    Making the mirror operation, i.e. applying property (ii) from \Cref{partition-lemma} to the second term and disjointedness of $E(\bar{b})$'s to the first one gives
    \begin{align*}
     \bar{\mu}_{\bar{a}}\left(\bigcup_{n=1}^{\infty}A_{n}\right) &   = \sum_{\bar{c}\in L_{0}}\sum_{\bar{b}\in N_{\bar{c}}} \bar{\mu}_{\bar{a}}(E(\bar{b})) + \sum_{\bar{b}\in N_{0}} \bar{\mu}_{\bar{a}}(E(\bar{b}))
     \end{align*}
     Finally, using that $\{N_{0}\} \cup \{N_{\bar{c}}\}_{\bar{c}\in L_{0}}$ partitions $N:=\bigcup_{i=1}^{\infty}K_{i}$ and (again) the disjointedness of $E(\bar{b})$'s, one obtains
     \begin{align*}
       \bar{\mu}_{\bar{a}}\left(\bigcup_{n=1}^{\infty}A_{n}\right) &= \sum_{\bar{b}\in N} \bar{\mu}_{\bar{a}}(E(\bar{b})) =\sum_{\bar{b}\in \cup_{i\in \N}K_{i}}\bar{\mu}_{\bar{a}}(E(\bar{b}))= \sum_{i=1}^{\infty} \sum_{\bar{b}\in K_{i}}\bar{\mu}_{\bar{a}}(E(\bar{b})) =\sum_{i=1}^{\infty} \bar{\mu}_{\bar{a}}(\cup_{\bar{b}\in K_{i}}E(\bar{b}))=\sum_{i=1}^{\infty} \bar{\mu}_{\bar{a}}(A_{i}).
    \end{align*}
   Note that all of the above sums were over countable index sets since they're subsets of a countable set $\mathcal{O}^{\ast}$. Moreover, all of the rearrangements are valid since the series in question converge absolutely. Therefore, it can be concluded that $\bar{\mu}_{\bar{a}}$ is a pre-measure on a Boolean algebra $\mathcal{E}$. Since for $\bar{a}\in\O^{\ast}$ with $\P(\bar{a})\neq 0$, it holds that
    \begin{align*}
        \bar{\mu}_{\bar{a}}(F) =\bar{\mu}_{\bar{a}}\left(\bigcup_{b\in \O} E(b)\right) = \sum_{b\in \O}\bar{\mu}_{\bar{a}}(E(b))= \sum_{b\in\O}\P(b\mid\bar{a})= 1,
    \end{align*}
    while $\bar{\mu}_{\bar{a}}(F)=0$ for $\bar{a}$ with $\P(\bar{a})=0$, it follows that the pre-measure $\mu_{\bar{a}}$ is finite.
\end{proof}

Using the previous two lemmas the following important result can be deduced.

\begin{theorem}\label{pre_to_full-thm}
        For any $\bar{a}\in \O^{\ast}$, there is unique finite measure $\mu_{\bar{a}}$ on the $\sigma$-algebra $\sigma(\mathcal{E})$ generated by $\mathcal{E}$ that extends the pre-measure $\bar{\mu}_{\bar{a}}$. It is given by
        \begin{align*}
            \mu_{\bar{a}}(A) = \inf\left\{\sum_{i=1}^{\infty}\bar{\mu}_{\bar{a}}(A_{i}) \mid A\subseteq \bigcup_{i=1}^{\infty}A_{i};\; A_{i}\in \mathcal{E} \text{ for each } i\in\mathbb{N}\right\}
        \end{align*}
\end{theorem}
\begin{proof}
    Lemmas \ref{algebra} and \ref{premeasure} allow us to apply the Hahn-Kolmogorov theorem and arrive at the conclusion.
\end{proof}

Thus for each $\bar{a}\in\O^{\ast}$ there is a measure $\mu_{\bar{a}}$ on $\sigma(\mathcal{E})$. Observe that one of the elements of $\O^{\ast}$ is special, namely the empty string $\epsilon$. Hence we also have a distinguished measure $\mu_{\epsilon}$. Using it, each of the other measures $\mu_{\bar{a}}$ can be expressed as a density with respect to the special $\mu_{\epsilon}$ measure. More specifically, consider the following.
\begin{lemma}
    For $\bar{a}\in\O^{\ast}$ such that $\P(\bar{a})\neq 0$, the measure $\mu_{\bar{a}}$ is majorized by $\frac{1}{\P(\bar{a})}\mu_{\epsilon}$, i.e. for any $A\in \sigma(\mathcal{E})$ it holds that $\mu_{\bar{a}}(A) \leq  \frac{1}{\P(\bar{a})}\mu_{\epsilon}(A)$. 
\end{lemma}
\begin{proof}
    First, observe that $\frac{1}{\P(\bar{a})}\mu_{\epsilon}$ is indeed a measure as a positive multiple of measure $\mu_{\epsilon}$. Now take any $S\subseteq \mathcal{O}^{k}$, then
    \begin{align*}
        \mu_{\bar{a}}\left(\bigcup_{\bar{b}\in S}E(\bar{b})\right) = \sum_{\bar{b}\in S}\mu_{\bar{a}}(E(\bar{b})) =  \frac{1}{\P(\bar{a})} \sum_{\bar{b}\in S}\P(\bar{a}\bar{b})\leq \frac{1}{\P(\bar{a})} \sum_{\bar{b}\in S}\P(\bar{b}) = \frac{1}{\P(\bar{a})} \sum_{\bar{b}\in S}\mu_{\epsilon}(\bar{b}) = \frac{1}{\P(\bar{a})} \mu_{\epsilon}\left(\bigcup_{\bar{b}\in S}E(\bar{b})\right).
    \end{align*}
    In other words,
    \begin{align*}
        \mu_{\bar{a}}(B) \leq  \frac{1}{\P(\bar{a})}\mu_{\epsilon}(B)
    \end{align*}
    for any $B\in \mathcal{E}$. Finally, take any $A\in \sigma(\mathcal{E})$ and consider that
    \begin{align*}
        \mu_{\bar{a}}(A) &= \inf\left\{\sum_{i=1}^{\infty}\bar{\mu}_{\bar{a}}(A_{i}) \mid A\subseteq \bigcup_{i=1}^{\infty}A_{i};\; A_{i}\in \mathcal{E} \text{ for each } i\in\mathbb{N}\right\}\\ &\leq \inf\left\{\frac{1}{\P(\bar{a})}\sum_{i=1}^{\infty}\bar{\mu}_{\epsilon}(A_{i}) \mid A\subseteq \bigcup_{i=1}^{\infty}A_{i};\; A_{i}\in \mathcal{E} \text{ for each } i\in\mathbb{N}\right\} = \frac{1}{\P(\bar{a})}\mu_{\epsilon}(A),
    \end{align*}
    where it was used that, by construction, $\bar{\mu}$ and $\mu$ agree on $\mathcal{E}$. This concludes the proof.
\end{proof}

Since for $\P(\bar{a})\neq 0$, ${\mu}_{\bar{a}}$  is majorized by $\frac{1}{\P(\bar{a})} {\mu}_{\epsilon}$ and both of these measures are finite, the Radon-Nikodym theorem (in majorization case) can be applied to conclude that there exists a ${\mu}_\epsilon$-measurable\footnote{Recall that for $k>0$, a function $f$ is $\mu$-measurable if and only if it is $k\mu$-measurable, and $\int f d(k\mu) = k \int fd\mu$} function $h_{\bar{a}}$ with $0\leq h_{\bar{a}}\leq1$ almost everywhere with respect to $\mu_{\epsilon}$, such that
\begin{align}
    {\mu}_{\bar{a}}(A) = \frac{1}{\P(\bar{a})}\int_{A}h_{\bar{a}} d{\mu}_{\epsilon}.
\end{align}
This function $h_{\bar{a}}$ is uniquely determined up to sets of ${\mu}_{\epsilon}$-measure $0$. From now on, sets of measure zero and properties holding almost everywhere are always with respect to $\mu_{\epsilon}$. By defining
\begin{align*}
    \gamma(g_{\bar{a}}) :=  \frac{1}{\P(\bar{a})}h_{\bar{a}},
\end{align*}
it follows that 
\begin{align*}
    {\mu}_{\bar{a}}(A) = \int_{A}\gamma(g_{\bar{a}}) d{\mu}_{\epsilon},
\end{align*}
and $0\leq \gamma(g_{\bar{a}})\leq \frac{1}{\P(\bar{a})}$ almost everywhere. Moreover, $\gamma(g_{\bar{a}})$ is unique up to sets of measure $0$. Finally, observe that if $\P(\bar{a})=0$ then $\mu_{\bar{a}} \equiv 0$ and so by defining $\gamma(g_{\bar{a}}) \equiv 0$ it follows that
\begin{align*}
    \mu_{\bar{a}}(A) = \int_{A}\gamma(g_{\bar{a}}) d\mu_{\epsilon}.
\end{align*}
This density is unique up to sets of measure zero. As a consequence, the following holds.

\begin{theorem}
    For each $\bar{a}\in\O^{\ast}$, $\gamma(g_{\bar{a}})\in \mathcal{L}^{\infty}(\mu_{\epsilon})$.
\end{theorem}

\begin{proof}
    If $\P(\bar{a})=0$ then $\gamma(g_{\bar{a}})\equiv 0$. While for $\P(\bar{a})\neq 0$, it holds that $\gamma(g_{\bar{a}}) \leq \frac{1}{\P(\bar{a})}$ almost everywhere. Therefore in both cases $\gamma(g_{\bar{a}})$ is uniformly bounded almost everywhere and so $\gamma(g_{\bar{a}})\in \mathcal{L}^{\infty}(\mu_{\epsilon})$
\end{proof}

 \begin{corollary}
     For each $q\in[1,\infty)$, $\gamma(g_{\bar{a}})\in \mathcal{L}^{q}(\mu_{\epsilon})$.
 \end{corollary}
\begin{proof}
    This follows from the above theorem since
    \begin{align*}
        \int_{F} |\gamma(g_{\bar{a}})|^{q} d\mu_{\epsilon} \leq \norm{\gamma(g_{\bar{a}})}_{\infty}^{q} \mu_{\epsilon}(F) <\infty.
    \end{align*}
\end{proof}

In particular, for each $\bar{a}\in \Lambda$ we have that $\gamma(g_{\bar{a}})\in\mathcal{L}^{2}(\mu_{\epsilon})$. Moreover, by the uniqueness part of the Radon-Nikodym Theorem, to each $\bar{a}\in\Lambda$ one can associate a unique element $[\gamma(g_{\bar{a}})]$ of $L^{2}(\mu_{\epsilon})$, where $L^{2}(\mu_{\epsilon})$ is the quotient of $\mathcal{L}^{2}(\mu_{\epsilon})$ by the equivalence relation that identifies functions that are equal almost everywhere. Thus $[\gamma(g_{\bar{a}})]$ is the equivalence class of densities which are the same almost everywhere. Uniqueness of $[\gamma(g_{\bar{a}})]$ makes the following map well-defined. Let $\phi:\mathcal{G}\to L^{2}(\mu_{\epsilon})$ be the linear map defined by
\begin{align*}
    \phi(g_{\bar{a}}) := [\gamma(g_{\bar{a}})]
\end{align*}
for $\bar{a}\in\Lambda$ and then extend linearly to other elements of $\mathcal{G}$. A representative of $\phi(g)$ will be denoted by $\gamma(g)$.

\begin{theorem}
    The linear map $\phi:\mathcal{G}\rightarrow L^{2}(\mu_{\epsilon})$ is an embedding and its definition does not depend on the choice of basis $\Lambda$. 
\end{theorem}

\begin{proof}
    To prove basis independence, recall that a necessary condition for $g_{\bar{b}}$ to be part of a basis is that $\P(\bar{b})>0$, otherwise $g_{\bar{b}}\equiv 0$. Thus, take any $g_{\bar{b}}$ with $\P(\bar{b})>0$ and suppose that $g_{\bar{b}}=\sum_{\bar{a}\in \Lambda} \alpha_{\bar{a}}g_{\bar{a}}$. To prove basis independence, the following must be proven
    \begin{align*}
        [\gamma(g_{\bar{b}})] = \sum_{\bar{a}\in \Lambda} \alpha_{\bar{a}}[\gamma(g_{\bar{a}})].
    \end{align*}
    To this end\footnote{Recall that addition and multiplication on $L^{2}$ have the property that $[\alpha f+\alpha g]=\alpha[f]+\beta[g]$ for $\alpha,\beta\in\R$. Thus, to prove the above equality, it suffices to check that it holds almost everywhere for some representatives of the equivalence classes. This follows since if such almost everywhere equality is proven, the corresponding equivalence classes are the same and one can use $[f+\alpha g]=[f]+\alpha[g]$ to simplify the final answer into the desired form.}, first observe that writing $g_{\bar{b}}=\sum_{\bar{a}\in \Lambda} \alpha_{\bar{a}}g_{\bar{a}}$ and evaluating it at a set $A\in \mathcal{E}$ gives
    \begin{align*}
        \mu_{\bar{b}}(A)=\mu_{\bar{b}}\left(\bigcup_{\bar{c}\in K}E(\bar{c})\right)=\sum_{\bar{c}\in K}g_{\bar{b}}(\bar{c})=\sum_{\bar{c}\in K}\sum_{\bar{a}\in \Lambda} \alpha_{\bar{a}}g_{\bar{a}}(\bar{c}) = \sum_{\bar{a}\in \Lambda} \alpha_{\bar{a}}\sum_{\bar{c}\in K}g_{\bar{a}}(\bar{c})=\sum_{\bar{a}\in \Lambda} \alpha_{\bar{a}}\mu_{\bar{a}}(A),
    \end{align*}
    where to exchange the sums the fact that only finitely many $\alpha_{\bar{a}}$ are nonzero was used. This equality can be extended to all of $\sigma(\mathcal{E})$ since for $B\in \sigma(\mathcal{E})$, we have that
    \begin{align*}
        \mu_{\bar{b}}(B)  &=\inf\left\{\sum_{i=1}^{\infty}\bar{\mu}_{\bar{a}}(A_{i}) \mid B\subseteq \bigcup_{i=1}^{\infty}A_{i};\; A_{i}\in \mathcal{E} \text{ for each } i\in\mathbb{N}\right\} \\&= \inf\left\{\sum_{i=1}^{\infty}\sum_{\bar{a}\in \Lambda} \alpha_{\bar{a}}\bar{\mu}_{\bar{a}}(A_{i}) \mid B\subseteq \bigcup_{i=1}^{\infty}A_{i};\; A_{i}\in \mathcal{E} \text{ for each } i\in\mathbb{N}\right\} = \sum_{\bar{a}\in \Lambda} \alpha_{\bar{a}}\mu_{\bar{a}}(B),
    \end{align*}
    where it was used that $\bar{\mu}_{\bar{a}}$ and $\mu_{\bar{a}}$ agree on $\mathcal{E}$. Now, recall that
    \begin{align*}
        \mu_{\bar{b}}(B) = \int_{B}\gamma(g_{\bar{b}})d\mu_{\epsilon}.
    \end{align*}
    On the other hand, it also holds that
    \begin{align*}
        \mu_{\bar{b}}(B)=\sum_{\bar{a}\in \Lambda} \alpha_{\bar{a}}\mu_{\bar{a}}(B)  =\sum_{\bar{a}\in \Lambda} \alpha_{\bar{a}}\int_{B}\gamma(g_{\bar{a}})d\mu_{\epsilon}   =\int_{B}\sum_{\bar{a}\in \Lambda} \alpha_{\bar{a}}\gamma(g_{\bar{a}})d\mu_{\epsilon},
    \end{align*}
    where the linearity of the integral and that only finitely many $\alpha_{\bar{a}}$ are nonzero were used. By the uniqueness part of the Radon-Nikodym theorem, it follows that $\gamma(g_{\bar{b}}) = \sum_{\bar{a}\in \Lambda} \alpha_{\bar{a}}\gamma(g_{\bar{a}})$ almost everywhere. This proves the claim that 
    \begin{align*}
        [\gamma(g_{\bar{b}})] = \sum_{\bar{a}\in \Lambda} \alpha_{\bar{a}}[\gamma(g_{\bar{a}})].
    \end{align*}
    Hence $\phi(g_{\bar{b}})=[\gamma(g_{\bar{b}})]$ for any element $\bar{b}\in \O^{\ast}$ with $\P(\bar{b})>0$. Consequently, the definition of $\phi$ does not depend on the particular choice of basis $\Lambda$. What remains to be shown is that $\phi$ is injective. So suppose $\phi(g)=0$ for $g=\sum_{\bar{a}\in \Lambda}\alpha_{\bar{a}}g_{\bar{a}}$, hence
    \begin{align*}
        \phi(g)=\sum_{\bar{a}\in \Lambda} \alpha_{\bar{a}}[\gamma(g_{\bar{a}})] = 0.
    \end{align*}
    Picking a representative of each basis element gives
    \begin{align*}
        \sum_{\bar{a}\in \Lambda} \alpha_{\bar{a}}\gamma(g_{\bar{a}}) = 0
    \end{align*}
    almost everywhere. From this, it follows that
    \begin{align*}
        g(\bar{b}) = \sum_{\bar{a}\in \Lambda} \alpha_{\bar{a}}g_{\bar{a}}(\bar{b}) = \sum_{\bar{a}\in \Lambda} \alpha_{\bar{a}}\mu_{\bar{a}}(E(\bar{b})) =  \sum_{\bar{a}\in \Lambda} \alpha_{\bar{a}}\int_{E(\bar{b})}\gamma(g_{\bar{a}})d\mu_{\epsilon} = \int_{E(\bar{b})}\sum_{\bar{a}\in \Lambda} \alpha_{\bar{a}}\gamma(g_{\bar{a}})d\mu_{\epsilon} = 0.
    \end{align*}
     Since $\bar{b}$ was an arbitrary element, it must be that $g \equiv 0$ and hence $\phi$ is an embedding.
\end{proof}

 Since $\phi(\mathcal{G})$ is a linear subspace of the Hilbert space $L^{2}(\mu_{\epsilon})$, it inherits its inner product structure which, in turn, can be endowed to $\mathcal{G}$ using its identification with $\phi(\mathcal{G})$. To be explicit, there is the following result.
 \begin{theorem}
     For $g_{1}= \sum_{\bar{a}\in \Lambda} \alpha_{\bar{a}}g_{\bar{a}}\in \mathcal{G}$ and $g_{2}=\sum_{\bar{b}\in \Lambda} \beta_{\bar{b}}g_{\bar{b}}\in\mathcal{G}$, the map 
     \begin{align*}
         \langle g_{1}, g_{2} \rangle := \langle \phi(g_{1}), \phi(g_{2}) \rangle = \int_{F} \left(\sum_{\bar{a}\in \Lambda} \alpha_{\bar{a}}\gamma(g_{\bar{a}})\right)\left(\sum_{\bar{b}\in \Lambda} \beta_{\bar{b}}\gamma(g_{\bar{b}})\right)d\mu_{\epsilon}
     \end{align*}
     defines an inner product on $\mathcal{G}$.
 \end{theorem}

\section{Approximation Theory}

Note that the inner product on $\mathcal{G}$ is greatly implicit because the functions $\gamma(g)$ are constructed using the Radon-Nikodym theorem. Fortunately, both the functions $\gamma(g)$ and the inner product can be computed using the following results.

\begin{theorem} \label{approx_thm}Suppose that for all $\bar{c}\in\mathcal{O}^{\ast}$ and almost-all $\tilde{a}\in F$ the limit 
\begin{align*}
    \lim_{n\to\infty} \frac{g_{\bar{c}}(\tilde{a}^{k})}{g_{\epsilon}(\tilde{a}^{k})}
\end{align*}
exists. Then for $g,g_{1},g_{2}\in\mathcal{G}$, $\tilde{b}\in F$ and $\bar{c},\bar{d}\in\O^{\ast}$ the following hold:
    \begin{enumerate}[label=(\roman*)]
        \item The function $\gamma(g)$ can be computed almost everywhere by 
        \[\gamma(g)(\tilde{b})=\lim_{k\to\infty}\frac{g(\tilde{b}^{k})}{g_{\epsilon}(\tilde{b}^{k})}.\]
        \item 
        \[
            \innp{g_{1}}{g_{2}} = \lim_{n\to \infty}\sum_{\bar{a}\in\O^{n}} \frac{g_{1}(\bar{a})g_{2}(\bar{a})}{g_{\epsilon}(\bar{a})}.
        \]
        \item If $P(\bar{c})\neq 0$ and $\P(\bar{d})\neq 0$ then
        \[
            \innp{g_{\bar{c}}}{g_{\bar{d}}} = \frac{1}{\sigma(t_{\bar{c}}g_{\epsilon})\sigma(t_{\bar{d}}g_{\epsilon})}\lim_{n\to \infty}\sum_{\bar{a}\in\O^{n}} \frac{\sigma(t_{\bar{c}\bar{a}}g_{\epsilon})\sigma(t_{\bar{d}\bar{a}}g_{\epsilon})}{\sigma(t_{\bar{a}}g_{\epsilon})},
       \]
        where for $\bar{b}=b_{1}b_{2}\ldots b_{k}\in\O^{k}$ the operator $t_{\bar{b}}$ is defined to be the composition $t_{b_{k}}\ldots t_{b_{2}}t_{b_{1}}$.
    \end{enumerate}
\end{theorem}

\begin{proof}
    \begin{enumerate}[label=(\roman*)]
    \item Take $\bar{c}\in \Lambda$ and define function $f$ to be
    \begin{align*}
        f(\tilde{a}) := \lim_{n\to\infty} \frac{g_{\bar{c}}(\tilde{a}^{k})}{g_{\epsilon}(\tilde{a}^{k})},
    \end{align*}
    wherever the limit exists and $f(\tilde{a}):=0$ otherwise. Observe that the following holds almost everywhere
    \begin{align*}
        f(\tilde{b}) &= \lim_{k\to \infty}  \frac{g_{\bar{c}}(\tilde{b}^{k})}{g_{\epsilon}(\tilde{b}^{k})}= \lim_{k\to \infty} \sum_{\bar{a}\in \O^{k}} \frac{g_{\bar{c}}(\bar{a})}{g_{\epsilon}(\bar{a})} \mathds{1}_{E(\bar{a})} (\tilde{b}).
    \end{align*}
    Thus it follows that
    \begin{align*}
        f_{k}(\tilde{b}) :=\sum_{\bar{a}\in \O^{k}} \frac{g_{\bar{c}}(\bar{a})}{g_{\epsilon}(\bar{a})} \mathds{1}_{E(\bar{a})} (\tilde{b}) \xrightarrow{k\to\infty} f(\tilde{b})
    \end{align*}
    almost everywhere and $f_{k}$ are integrable since they're simple functions. Next, observe that
    \begin{align*}
        |f_{k}(\tilde{b})| &= \left| \sum_{\bar{a}\in \O^{k}} \frac{g_{\bar{c}}(\bar{a})}{g_{\epsilon}(\bar{a})} \mathds{1}_{E(\bar{a})} (\tilde{b})\right| = \left| \frac{g_{\bar{c}}(\tilde{b}^{k})}{g_{\epsilon}(\tilde{b}^{k})} \right| = \frac{\P(\bar{c}\tilde{b}^{k})}{\P(\tilde{b}^{k})\P(\bar{c})} ,
    \end{align*}
    can be bounded by
    \begin{align*}
        |f_{k}(\tilde{b})|&\leq  \frac{1}{\P(\bar{c})} =: G,
    \end{align*}
    which is finite since $g_{\bar{c}}$ is a basis element and so $\P(\bar{c})\neq 0$. Since $G$ is a constant and $\mu_{\epsilon}$ is finite, it follows that $G$, as a constant function, is integrable. As $G$ dominates $f_{k}$'s, the Dominated Convergence Theorem can be applied to conclude that $f$ is integrable and for any $B\in \sigma(\mathcal{E})$
    \begin{align*}
        \int_{B} f d\mu_{\epsilon} &= \lim_{k\to \infty} \int_{B}f_{k}d\mu_{\epsilon}.
    \end{align*}
    Since $f$ is nonnegative, it follows that
    \begin{align*}
        \mu(B) := \int_{B} f d\mu_{\epsilon} 
    \end{align*}
    defines a measure on $\sigma(\mathcal{E})$. Now suppose $B\in \mathcal{E}$ and apply the definition of $f_{k}$ to get
    \begin{align*}
         \mu(B) = \int_{B} f d\mu_{\epsilon} &=\lim_{k\to \infty}  \int_{B}\sum_{\bar{a}\in \O^{k}} \frac{g_{\bar{c}}(\bar{a})}{g_{\epsilon}(\bar{a})} \mathds{1}_{E(\bar{a})}d\mu_{\epsilon}=\lim_{k\to \infty} \sum_{\bar{a}\in \O^{k}} \frac{g_{\bar{c}}(\bar{a})}{g_{\epsilon}(\bar{a})} \mu_{\epsilon}(E(\bar{a}) \cap B).
    \end{align*}
    But, for $B\in \mathcal{E}$ it holds that
    \begin{align*}
         \mu(B) =\lim_{k\to \infty} \sum_{\bar{a}\in \O^{k}} \frac{g_{\bar{c}}(\bar{a})}{g_{\epsilon}(\bar{a})} \mu_{\epsilon}\left(E(\bar{a}) \cap \bigcup_{\bar{b}\in K} E(\bar{b})\right)=\lim_{k\to \infty} \sum_{\bar{a}\in \O^{k}} \frac{g_{\bar{c}}(\bar{a})}{g_{\epsilon}(\bar{a})} \mu_{\epsilon}\left( \bigcup_{\bar{b}\in K} E(\bar{a}) \cap E(\bar{b})\right).
    \end{align*}
    For $k>n$ where $n$ is such that $K\subseteq \mathcal{O}^{n}$, it holds that 
    \begin{align*}
        E(\bar{a}) \cap E(\bar{b}) =\begin{cases}
            E(\bar{a}) & \text{if } \bar{a}=\bar{b}\bar{w} \text{ for some sequence } \bar{w}\\
            \emptyset &\text{otherwise.}
        \end{cases}
    \end{align*}
    Therefore, using disjointedness  of $E(\bar{b})$'s, it follows that
     \begin{align*}
         \mu(B) =\lim_{k\to \infty} \sum_{\substack{\bar{a}\in \O^{k+n}: \\ \bar{a}=\bar{b}\bar{w}\\\text{for some } \bar{b}\in K}} \frac{g_{\bar{c}}(\bar{a})}{g_{\epsilon}(\bar{a})} \mu_{\epsilon}\left(E(\bar{a})\right)= \lim_{k\to \infty} \sum_{\bar{b}\in K} \sum_{\substack{\bar{w}\in \O^{k}}} g_{\bar{c}}(\bar{b}\bar{w})= \lim_{k\to \infty} \sum_{\bar{b}\in K} \sum_{\substack{\bar{w}\in \O^{k}}} \P(\bar{b}\bar{w}\mid \bar{c}).
    \end{align*}
    Using that $\P(\bar{b}\bar{w}\mid \bar{c}) = \P(\bar{w}\mid\bar{c}\bar{b})\P(\bar{b}\mid \bar{c})$ gives
    \begin{align*}
         \mu(B)  = \lim_{k\to \infty} \sum_{\bar{b}\in K} \sum_{\substack{\bar{w}\in \O^{k}}} \P(\bar{w}\mid\bar{c}\bar{b})\P(\bar{b}\mid \bar{c}) = \sum_{\bar{b}\in K}\P(\bar{b}\mid \bar{c})=\sum_{\bar{b}\in K}g_{\bar{c}}(\bar{b}) = \mu_{\bar{c}}(B).
    \end{align*}
    Since are both $\mu$ and $\mu_{\bar{c}}$ are finite measures on $\sigma(\mathcal{E})$ that agree on the generating algebra $\mathcal{E}$, it follows by uniqueness of the Hahn-Kolmogorov extension theorem that $\mu=\mu_{\bar{c}}$. Consequently, for any $A\in \sigma(\mathcal{E})$
    \begin{align*}
        \mu_{\bar{c}}(A) = \mu(A) = \int_{A} f d\mu_{\epsilon}.
    \end{align*}
    From the uniqueness of the Radon-Nikodym theorem, it follows that
    \begin{align*}
        \gamma(g_{\bar{c}})(\tilde{a}) = f(\tilde{a}) =\lim_{n\to\infty} \frac{g_{\bar{c}}(\tilde{a}^{k})}{g_{\epsilon}(\tilde{a}^{k})}
    \end{align*}
    almost everywhere. Finally, take $g\in \mathcal{G}$ and write it as $g= \sum_{\bar{a}\in \Lambda}\alpha_{\bar{a}}g_{\bar{a}}$ then
    \begin{align*}
        \gamma(g)(\tilde{b}) = \sum_{\bar{a}\in\Lambda}\alpha_{\bar{a}}\gamma(g_{\bar{a}})(\tilde{b})= \sum_{\bar{a}\in\Lambda}\alpha_{\bar{a}}\lim_{k\to\infty}\frac{g_{\bar{a}}(\tilde{b}^{k})}{g_{\epsilon}(\tilde{b}^{k})}= \lim_{k\to\infty}\frac{\sum_{\bar{a}\in\Lambda}\alpha_{\bar{a}}g_{\bar{a}}(\tilde{b}^{k})}{g_{\epsilon}(\tilde{b}^{k})}= \lim_{k\to\infty}\frac{g(\tilde{b}^{k})}{g_{\epsilon}(\tilde{b}^{k})},
    \end{align*}
    where to interchange the sum and the limit it was used that only finitely many $\alpha_{\bar{a}}$ are nonzero. Since each limit exists almost everywhere and only finitely many of them are summed, it follows that the above equality holds almost everywhere.

    \item To prove this result, the same strategy is followed as above. First, using the previous result, observe that the following holds almost everywhere
    \begin{align*}
        \gamma(g_{1})(\tilde{b})\gamma(g_{2})(\tilde{b}) &= \lim_{k\to \infty}  \frac{g_{1}(\tilde{b}^{k})g_{2}(\tilde{b}^{k})}{g_{\epsilon}(\tilde{b}^{k})^{2}}= \lim_{k\to \infty} \sum_{\bar{a}\in \O^{k}} \frac{g_{1}(\bar{a})g_{2}(\bar{a})}{g_{\epsilon}(\bar{a})^{2}} \mathds{1}_{E(\bar{a})} (\tilde{b}).
    \end{align*}
    Thus it follows that
    \begin{align*}
        f_{k}(\tilde{b}) :=\sum_{\bar{a}\in \O^{k}} \frac{g_{1}(\bar{a})g_{2}(\bar{a})}{g_{\epsilon}(\bar{a})^{2}} \mathds{1}_{E(\bar{a})} (\tilde{b}) \xrightarrow{k\to\infty} (\gamma(g_{1})\gamma(g_{2}))(\tilde{b})
    \end{align*}
    almost everywhere and $f_{k}$ are integrable since they're simple functions. Next, observe that
    \begin{align*}
        |f_{k}(\tilde{b})| &= \left| \sum_{\bar{a}\in \O^{k}} \frac{g_{1}(\bar{a})g_{2}(\bar{a})}{g_{\epsilon}(\bar{a})^{2}} \mathds{1}_{E(\bar{a})} (\tilde{b})\right| = \left| \frac{g_{1}(\tilde{b}^{k})g_{2}(\tilde{b}^{k})}{g_{\epsilon}(\tilde{b}^{k})^{2}} \right|.
    \end{align*}
    Writing $g_{1}=\sum_{\bar{a}\in\Lambda}\alpha_{\bar{a}}g_{\bar{a}}$ and $g_{2}=\sum_{\bar{c}\in \Lambda}\beta_{\bar{c}}g_{\bar{c}}$ gives
    \begin{align*}
         |f_{k}(\tilde{b})|&= \left| \frac{\sum_{\bar{a}\in\Lambda}\alpha_{\bar{a}}g_{\bar{a}}(\tilde{b}^{k})}{g_{\epsilon}(\tilde{b}^{k})} \right|\left| \frac{\sum_{\bar{c}\in\Lambda}\beta_{\bar{c}}g_{\bar{c}}(\tilde{b}^{k})}{g_{\epsilon}(\tilde{b}^{k})} \right|= \left| \frac{\sum_{\bar{a}\in\Lambda}\alpha_{\bar{a}}\P(\bar{a}\tilde{b}^{k} )/\P(\bar{a})}{\P(\tilde{b}^{k})} \right|\left| \frac{\sum_{\bar{c}\in\Lambda}\beta_{\bar{c}}\P(\bar{c}\tilde{b}^{k})/\P(\bar{c})}{\P(\tilde{b}^{k})} \right|.
    \end{align*}
    Now the triangle inequality can be applied to obtain
    \begin{align*}
         |f_{k}(\tilde{b})|&\leq  \frac{\sum_{\bar{a}\in\Lambda}\left|\alpha_{\bar{a}} \right|\P(\bar{a}\tilde{b}^{k})/\P(\bar{a})}{\P(\tilde{b}^{k})} \frac{\sum_{\bar{c}\in\Lambda}\left| \beta_{\bar{c}}\right|\P(\bar{c}\tilde{b}^{k})/\P(\bar{c})}{\P(\tilde{b}^{k})}
    \end{align*}
    Finally, using that $\P(\bar{a}\tilde{b}^{k}) \leq \P(\tilde{b}^{k})$ gives
    \begin{align*}
        |f_{k}(\tilde{b})|&\leq  \sum_{\bar{a}\in\Lambda}\frac{\left|\alpha_{\bar{a}} \right|}{g_{\epsilon}(\bar{a})} \sum_{\bar{c}\in\Lambda}\frac{\left| \beta_{\bar{c}}\right|}{g_{\epsilon}(\bar{a})} =: G_{12},
    \end{align*}
    which is finite since there are only finitely many nonzero coefficients $\alpha_{\bar{a}},\beta_{\bar{a}}$. Since $G_{12}$ is a constant and $\mu_{\epsilon}$ is finite, it follows that $G_{12}$ --- viewed as a constant function --- is integrable. As $G_{12}$ dominates $f_{k}$'s, the Dominated Convergence Theorem can be applied to conclude that
    \begin{align*}
        \innp{g_{1}}{g_{2}}=\int_{F} \gamma(g_{1})\gamma(g_{2})d\mu_{\epsilon} &= \lim_{k\to \infty} \int_{F}f_{k}d\mu_{\epsilon}.
    \end{align*}
    Finally, applying the definition of $f_{k}$ gives
    \begin{align*}
        \innp{g_{1}}{g_{2}}&=\lim_{k\to \infty}  \int_{F}\sum_{\bar{a}\in \O^{k}} \frac{g_{1}(\bar{a})g_{2}(\bar{a})}{g_{\epsilon}(\bar{a})^{2}} \mathds{1}_{E(\bar{a})}d\mu_{\epsilon}=\lim_{k\to \infty} \sum_{\bar{a}\in \O^{k}} \frac{g_{1}(\bar{a})g_{2}(\bar{a})}{g_{\epsilon}(\bar{a})^{2}} \mu_{\epsilon}(E(\bar{a}))=\lim_{k\to \infty} \sum_{\bar{a}\in \O^{k}} \frac{g_{1}(\bar{a})g_{2}(\bar{a})}{g_{\epsilon}(\bar{a})},
    \end{align*}
    which concludes the proof.
    \item Applying the results from (ii) and applying the definition of $g_{\bar{a}}$ yields 
    \begin{align*}
        \innp{g_{\bar{c}}}{g_{\bar{d}}} &:= \lim_{k\to \infty} \sum_{\bar{a}\in \O^{k}} \frac{g_{\bar{c}}(\bar{a})g_{\bar{d}}(\bar{a})}{g_{\epsilon}(\bar{a})}
        = \lim_{k\to \infty} \sum_{\bar{a}\in \O^{k}} \frac{\P(\bar{c}\bar{a})\P(\bar{d}\bar{a})}{\P(\bar{a})\P(\bar{c})\P(\bar{d})}.
    \end{align*}
    Using \Cref{OOM_Fun - thm} to write $\P(\bar{b})$ as $\sigma(t_{\bar{b}}g_{\epsilon})$ allows us to conclude that
    \begin{align*}
                \innp{g_{\bar{c}}}{g_{\bar{d}}} &=\frac{1}{\sigma(t_{\bar{c}}g_{\epsilon})\sigma(t_{\bar{d}}g_{\epsilon})}\lim_{k\to \infty} \sum_{\bar{a}\in \O^{k}} \frac{\sigma(t_{\bar{c}\bar{a}}g_{\epsilon})\sigma(t_{\bar{d}\bar{a}}g_{\epsilon})}{\sigma(t_{\bar{a}}g_{\epsilon})}.
    \end{align*}
       \end{enumerate}
\end{proof}

Under the assumptions of the preceding theorem, the observable operators $t_{a}$ are continuous with respect to the norm induced by the inner product. To prove this statement the following lemma will be needed.

\begin{lemma}\label{crucial-lemma}
    For any $g\in\mathcal{G}, a\in \mathcal{O}$ and $\bar{w}\in \mathcal{O}^{\ast}$, it holds that $(t_{a}g)(\bar{w})=g(a\bar{w})$.
\end{lemma}
\begin{proof}
    Writing $g= \sum_{\bar{a}\in\Lambda}\alpha_{\bar{a}}g_{\bar{a}}$, explicit computation yields
    \begin{align*}
    t_{a}(g)(\bar{w}) &= \sum_{\bar{c}\in\Lambda} \alpha_{\bar{c}} t_{a}(g_{\bar{c}})(\bar{w})= \sum_{\substack{\bar{c}\in\Lambda,\\ \P(\bar{c}a)\neq 0}} \lambda_{\bar{c}} \P(a\mid\bar{c})\P(\bar{w}\mid \bar{c}a)= \sum_{\substack{\bar{c}\in\Lambda,\\ \P(\bar{c}a)\neq 0}} \lambda_{\bar{c}} \P(a\bar{w}\mid\bar{c}).
\end{align*}
But if $\P(\bar{c}a)=0$ then $\P(a\bar{w}\mid \bar{c})=0$ and so
\begin{align*}
    t_{a}(g_{\bar{b}})(\bar{w})=\sum_{\substack{\bar{c}\in\Lambda}} \lambda_{\bar{c}} \P(a\bar{w}\mid\bar{c})=\sum_{\substack{\bar{c}\in\Lambda}} \lambda_{\bar{c}} g_{\bar{c}}(a\bar{w})=g(a\bar{w}).
\end{align*}
\end{proof}
\begin{theorem}\label{cont-thm}
    Under the assumption of \Cref{approx_thm}, the observable operators $t_{a}$ are continuous with respect to the 2-norm induced by the inner product on $\mathcal{G}$. Additionally, $\norm{t_{a}}\leq 1$.
\end{theorem}
\begin{proof}
    This follows from the following computation
    \begin{align*}
        \norm{t_{a}g}^{2} = \lim_{n\to \infty} \sum_{\bar{a}\in \mathcal{O}^{n}}\frac{(t_{a}g(\bar{a}))^{2}}{g_{\epsilon}(\bar{a})} = \lim_{n\to \infty} \sum_{\bar{a}\in \mathcal{O}^{n}}\frac{(g(a\bar{a}))^{2}}{g_{\epsilon}(\bar{a})} \leq \lim_{n\to \infty} \sum_{\bar{a}\in \mathcal{O}^{n}}\frac{(g(a\bar{a}))^{2}}{g_{\epsilon}(a\bar{a})} \leq \lim_{n\to \infty} \sum_{\bar{b}\in \mathcal{O}^{n+1}}\frac{(g(\bar{b}))^{2}}{g_{\epsilon}(\bar{b})} = \norm{g}^{2},
    \end{align*}
    where the last inequality holds because
    \begin{align*}
        \sum_{\bar{a}\in \mathcal{O}^{n}}\frac{(g(a\bar{a}))^{2}}{g_{\epsilon}(a\bar{a})} \leq \sum_{\bar{a}\in \mathcal{O}^{n}}\frac{(g(a\bar{a}))^{2}}{g_{\epsilon}(a\bar{a})} + \sum_{\substack{\bar{b}\in \mathcal{O}^{n+1}\\ \bar{b}\neq a\bar{a}}}\frac{(g(\bar{b}))^{2}}{g_{\epsilon}(\bar{b})} = \sum_{\bar{b}\in \mathcal{O}^{n+1}}\frac{(g(\bar{b}))^{2}}{g_{\epsilon}(\bar{b})}.
    \end{align*}
\end{proof}

Incidentally, \Cref{crucial-lemma} also gives the following result.

\begin{theorem}
     The observable operators $t_{a}$ are continuous with respect to the supremum norm on $\mathcal{G}$ and $\norm{t_{a}}\leq 1$.
\end{theorem}
\begin{proof}
    An explicit computation yields
    \begin{align*}
        \norm{t_{a}g}_{\infty}  = \sup_{\bar{w}\in\mathcal{O}^{\ast}}|t_{a}g(\bar{w})| = \sup_{\bar{w}\in\mathcal{O}^{\ast}}|g(a\bar{w})| \leq \sup_{\bar{w}\in\mathcal{O}^{\ast}}|g(\bar{w})| = \norm{g}_{\infty},
    \end{align*}
    which concludes the proof.
\end{proof}

\Cref{cont-thm} gives hopes that if the inner product space $\mathcal{G}$ was a Hilbert space, one could use one of the pathways outlined in the introduction and approximate the observable operators by finite-rank operators. The main original result of this thesis gives a full characterization of the conditions under which the space $\mathcal{G}$ is complete. Namely, we have the following.

\begin{theorem}
    Let $\norm{\cdot}$ be any norm on $\mathcal{G}$. Then $(\mathcal{G}, \norm{\cdot})$ is complete if and only if $\mathcal{G}$ is finite dimensional.
\end{theorem}
\begin{proof}
    The sufficient condition for $\mathcal{G}$ being complete follows from the well-known fact that any finite-dimensional normed vector space is complete.\\
    \indent To prove the other direction, let us assume that $\mathcal{G}$ is complete. Suppose towards contradiction that $\mathcal{G}$ is infinite-dimensional. Recalling that $\mathcal{G}$ was defined to be the linear span of functions $g_{\bar{a}}$ for $\bar{a}\in \O^{\ast}$, and that $\O^{\ast}$ is countable it follows that the (Hamel) dimension of $\mathcal{G}$ must be countable. However, as a consequence of Baire's Category Theorem, any infinite-dimensional complete normed vector space must have an uncountable (Hamel) basis \cite{hamelDim}. This yields a contradiction and consequently, $\mathcal{G}$ is finite-dimensional.
\end{proof}

\begin{corollary}
    The inner product space $\mathcal{G}$ is a Hilbert space if and only if the corresponding OOM is finite-dimensional.
\end{corollary}

Thus, in the infinite-dimensional case, the space $\mathcal{G}$ is simply too small to be complete, no matter the norm. Hence to develop the envisioned approximation theory for the observable operators, one needs to enlarge the space $\mathcal{G}$ while ensuring that the observable operators remain continuous. However, this is left for future research.
\section{Conclusions}

In conclusion, the research presented in this thesis explored a particular approach to developing an approximation theory of OOMs of infinite-dimensional processes. This approach was based on the observation that making the space of future distributions into a Hilbert space and proving the continuity of observable operators with respect to the associated 2-norm, would present the opportunity of following two distinct pathways towards an approximation theory. One through proving compactness of observable operators and the other through proving separability of the Hilbert space.

The research in this thesis focused on the first necessary steps of this approach. Namely, first, the inner product construction originally proposed in \cite{unpub-tutorial} has been put onto a mathematically rigorous footing and the missing proofs were rediscovered, including one that shows continuity of observable operators. Then an original theorem was proven which shows that an infinite-dimensional space of future distributions is too small to be a Hilbert space. Unfortunately, this result means that the original goal of making an infinite-dimensional space of future distributions into a Hilbert space cannot be achieved in the current framework. A possible solution to fixing it is to focus on the space $\phi(\mathcal{G})$ and the observable operators $\hat{t}_{a}:= \phi \circ t_{a} \circ \phi^{-1}$ on it\footnote{The map $\phi$ can be restricted to its range so that it becomes invertible.}. Then one could examine the closure of $\phi(\mathcal{G})$ in $L^{2}(\mu_{\epsilon})$ --- which is a Hilbert space --- and see whether the observable operators $\hat{t}_{a}$ can be continuously extended to observable operators $T_{a}$ on the Hilbert space $\overline{\phi(\mathcal{G})}$. If this is true, then approximating $T_{a}$ by finite-rank operators and restricting them to the subspace $\phi(\mathcal{G})$ would lead to a finite-rank approximation of $\hat{t}_{a}$ which could then be imported back to $\mathcal{G}$ using $\phi$. The exploration of this idea is left for future research.

\bibliographystyle{alpha}
\bibliography{literature}

\begin{appendices}
    
\section{Mathematical Background}\label{ap:appendix}

For the reader's convenience, the Hahn-Kolmogorov extension theorem, the Radon-Nikodym theorem and the required definitions are stated. The definitions used in the Hahn-Kolmogorv theorem and the theorem itself are based on \cite{tao-2011}, while the statement of the Radon-Nikodym theorem is adapted to the needs of this thesis and directly follows from the Radon-Nikodym theorem (and its proof) as presented in \cite{rudin-1987}. A reader in need of a measure theory recap or further details is referred to \cite{tao-2011}, \cite{bogachev2007measure} or \cite{rudin-1987}, while a brief introduction to $L^{p}$ spaces, which covers all their properties used in this thesis, can be found \cite{probText}. 

\begin{definition}
 Let $X$ be a set. A {\bf Boolean algebra} on $X$ is a collection $\mathcal{B}_{0}$ of subsets of $X$ which obeys the following properties:
 \begin{enumerate}
     \item (Empty set) $\emptyset\in\mathcal{B}_{0}$.
     \item (Complement) If $E\in \mathcal{B}_{0}$, then the complement $E^{c}:=X\setminus E$ also lies in $\mathcal{B}_{0}$.
     \item (Finite unions) If $E,F\in\mathcal{B}_{0}$, then $E\cup F\in \mathcal{B}_{0}$. 
 \end{enumerate}
\end{definition}

 \begin{definition}
     A {\bf pre-measure} on a Boolean algebra $\mathcal{B}_{0}$ is an extended real-valued function $\mu_{0}:\mathcal{B}_{0}\to[0,\infty]$ with the following properties:
     \begin{enumerate}
         \item $\mu_{0}(\emptyset)=0$.
         \item $\mu_{0}\left(\bigcup_{n=1}^{\infty}E_{n}\right)=\sum_{n=1}^{\infty}\mu_{0}(E_{n})$ whenever $E_{1},E_{2},...\in\mathcal{B}_{0}$ are disjoint sets such that $\bigcup_{n=1}^{\infty}E_{n}\in \mathcal{B}_{0}$.
     \end{enumerate}
 \end{definition}

\begin{definition}
    Let $X$ be a set. A {\bf $\sigma$-algebra} on $X$ is a collection $\mathcal{B}$ of subsets of $X$ which obeys the following properties:
    \begin{enumerate}
     \item (Empty set) $\emptyset\in\mathcal{B}$.
     \item (Complement) If $E\in \mathcal{B}$, then the complement $E^{c}:=X\setminus E$ also lies in $\mathcal{B}$.
     \item (Countable unions) If $E_{1},E_{2},...\in\mathcal{B}$, then $\bigcup_{i=1}^{\infty}E_{i}\in \mathcal{B}$. 
 \end{enumerate}
 We refer to the pair $(X,\mathcal{B})$ of a set $X$ together with a $\sigma$-algebra on that set as a {\bf measurable space}.
 \end{definition}

 \begin{definition}
     A {\bf measure} on a  $\sigma$-algebra $\mathcal{B}$ is an extended real-valued function $\mu:\mathcal{B}\to[0,\infty]$ with the following properties:
     \begin{enumerate}
         \item $\mu(\emptyset)=0$.
         \item $\mu\left(\bigcup_{n=1}^{\infty}E_{n}\right)=\sum_{n=1}^{\infty}\mu(E_{n})$ where $E_{1},E_{2},...\in\mathcal{B}$ are pairwise disjoint.
     \end{enumerate}
 \end{definition}

\begin{definition}
    A {\bf measure space} is a triple $(X,\mathcal{B},\mu)$ where $X$ is a set, $\mathcal{B}$ a $\sigma$-algebra on $X$ and $\mu$ a measure on $\mathcal{B}$. A measure space is {\bf finite} if $\mu(X)<\infty$. A measure space $\mu$ is {\bf $\sigma$-finite} if there is a sequence $X_{n}\in\mathcal{B}$ such that $\mu(X_{n})<\infty$ and $X=\bigcup_{i=1}^{\infty}X_{n}$. One also says that $\mu$ is finite/$\sigma$-finite with the measure space being implicit. A $\sigma$-finite pre-measure on a Boolean algebra is defined analogously.
\end{definition}

\begin{lemma}
    Let $\mathcal{E}$ be a collection of subsets of a set $X$. Then there exists a unique $\sigma$-algebra $\sigma(\mathcal{E})$ on $X$ such that
    \begin{enumerate}
        \item $\mathcal{E}\subset \sigma(\mathcal{E})$
        \item if $\mathcal{B}$ is $\sigma$-algebra with $\mathcal{E}\subset \mathcal{B}$, then $\sigma(\mathcal{E})\subset \mathcal{B}$.
    \end{enumerate}
    The unique $\sigma$-algebra $\sigma(\mathcal{E})$ is said to be {\bf generated by $\mathcal{E}$}.
\end{lemma}

 \begin{theorem}
     (Hahn-Kolmogorov theorem). Every pre-measure $\mu_{0}:\mathcal{B}_{0}\to [0,\infty]$ on a Boolean algebra $\mathcal{B}_{0}$ on $X$ can be extended to a measure $\mu:\sigma(\mathcal{B}_{0})\to[0,\infty]$ defined on the $\sigma$-algebra generated by $\mathcal{B}_{0}$. If $\mu_{0}$ is finite then the extension is unique and finite as well. Moreover, $\mu$ is given by
     \begin{align*}
            \mu(A) = \inf\left\{\sum_{i=1}^{\infty}{\mu}_{0}(A_{i}) \mid A\subseteq \bigcup_{i=1}^{\infty}A_{i};\; A_{i}\in \mathcal{B}_{0} \text{ for each } i\in\mathbb{N}\right\}
        \end{align*}
 \end{theorem}

 \begin{definition}
     Let $\mu$ and $\lambda$ be measures on a measurable space $(X,\mathcal{B})$. We say that $\lambda$ is {\bf majorized} by $\mu$ if for all $A\in\mathcal{B}$ it holds that $\lambda(A)\leq\mu(A)$.
 \end{definition}

 \begin{theorem}
     (Radon-Nikodym Theorem, majorization case). Suppose $\mu, \lambda$ are finite measures on a measurable space $(X,\mathcal{B})$ and that $\lambda$ is majorized by $\mu$. Then there exists a $\mu$-integrable real-valued function $f$ with $0\leq f\leq 1$ almost everywhere with respect to $\mu$, such that
     \begin{align*}
         \lambda(A) = \int_{A}fd\mu
     \end{align*}
     for every $A\in\mathcal{B}$. Moreover, such function $f$ is unique up to sets of $\mu$-measure $0$. 
 \end{theorem}
\end{appendices}

\end{document}